\documentclass[12pt]{amsart}
\usepackage[left=3.5cm,tmargin=2cm,bmargin=2.5cm,centering]{geometry}

\usepackage{hyperref}
\usepackage[utf8]{inputenc}
\usepackage[initials]{amsrefs}  
\usepackage{enumerate}
\usepackage{bm,amsmath,amssymb,amsthm}
\usepackage[all]{xy}
\usepackage{stmaryrd}
\usepackage{amsfonts}
\theoremstyle{plain}
\newtheorem{theorem}{Theorem}[section]

\newtheorem{lemma}[theorem]{Lemma}

\newtheorem*{proposition*}{Proposition}
\newtheorem*{lemma*}{Lemma}
\newtheorem*{corollary*}{Corollary}
\theoremstyle{definition}

\newtheorem*{remarks*}{Remarks}
\newtheorem*{remark*}{Remark}
\numberwithin{equation}{section}

\newcommand{\Mim}{\mathfrak{Im}\,}
\newcommand {\pnorm}[1]   {\left\lVert #1 \right\rVert}
\newcommand{\SB}{\backslash}
\newcommand{\FmH}{\mathfrak{H}}
\newcommand{\abs}[1]{\ensuremath{\left|#1\right|}}
\newcommand{\Lquote}[1]{``#1"}
\newcommand{\BmR}{\mathbb{R}}
\newcommand{\set}[1]{\left\{#1\right\}}
\newcommand{\Mdede}[4]{
\begin{pmatrix}
#1&#2  \\
#3&#4 \\
\end{pmatrix}
}
\newcommand{\BmZ}{\mathbb{Z}}
\newcommand{\CmF}{\mathcal{F}}

\newcommand{\C}{\mathbb{C}}

\begin{document}

\title[]{On the sup-norm of Maass cusp forms of large level. III}
\author[]{Gergely Harcos}
\address{Alfr\'ed R\'enyi Institute of Mathematics, Hungarian Academy of Sciences, POB 127, Budapest H-1364, Hungary}
\email{gharcos@renyi.hu}

\author[]{Nicolas Templier}   
\address{Department of Mathematics, Fine Hall, Washington Road, Princeton, NJ 08544-1000.}
\email{templier@math.princeton.edu}
\thanks{This work was partially supported by a grant \#209849 from the Simons Foundation (NT) and by EC grant ERG 239277 and OTKA grants K 72731, PD 75126 (GH)}
\date{Sep 2011}
\begin{abstract}
  Let $f$ be a Hecke--Maass cuspidal newform of square-free level $N$ and Laplacian eigenvalue $\lambda$. It is shown that $\pnorm{f}_\infty \ll_{\lambda,\epsilon} N^{-\frac{1}{6}+\epsilon} \pnorm{f}_2$ for any $\epsilon>0$.
\end{abstract}

\maketitle


\section{Introduction}\label{sec:intro}

This note deals with the problem of bounding the sup-norm of eigenfunctions on arithmetic hyperbolic surfaces. It is natural to restrict this problem to Hecke--Maass cuspidal newforms which are square-integrable joint eigenfunctions of the Laplacian and Hecke operators. We consider the noncompact modular surface $\Gamma_0(N)\SB \FmH$ equipped with its hyperbolic metric and associated measure; the total volume is then asymptotically equal to $N^{1+o(1)}$. We shall $L^2$-normalize all Hecke--Maass cuspidal newforms $f$ with respect to that measure, namely
\begin{equation}
\int_{\Gamma_0(N)\SB \FmH} \abs{f(z)}^2 \frac{dxdy}{y^2}=1.
\end{equation}
It is interesting to bound the sup-norm $\pnorm{f}_\infty$ in terms of the two basic parameters: the Laplacian eigenvalue $\lambda$ and the level $N$.

In the $\lambda$-aspect, the first nontrivial bound is due to Iwaniec and Sarnak~\cite{IS95} who established
$\pnorm{f}_\infty \ll_{N,\epsilon} \lambda^{\frac{5}{24}+\epsilon}$ for any $\epsilon>0$. Their key idea was to make use of the Hecke operators,
through the method of amplification, in order to go beyond $\pnorm{f}_\infty\ll_N \lambda^{\frac{1}{4}}$ which is valid on any Riemannian surface by~\cite{SS89}.

In the $N$-aspect, the \Lquote{trivial} bound is $\pnorm{f}_\infty \ll_{\lambda,\epsilon} N^{\epsilon}$ see~\cites{AU95,MU98,BH09}. Here and later the dependence on $\lambda$ is continuous. The first nontrivial bound in the $N$-aspect is due to Blomer--Holowinsky~\cite{BH09}*{p.~673} who proved $\pnorm{f}_\infty \ll_{\lambda,\epsilon} N^{-\frac{25}{914}+\epsilon}$, at least for square-free $N$. In~\cite{Temp:sup} the second named author revisited the proof by making a systematic use of geometric arguments, and derived a stronger exponent: $\pnorm{f}_\infty \ll_{\lambda,\epsilon} N^{-\frac{1}{22}+\epsilon}$. Helfgott--Ricotta (unpublished) improved some of the estimates in~\cite{Temp:sup} and obtained $\pnorm{f}_\infty \ll_{\lambda,\epsilon} N^{-\frac{1}{20}+\epsilon}$. In~\cite{Temp:harcos} we introduced a more efficient treatment of the counting problem at the heart of the argument and derived the estimate $\pnorm{f}_\infty \ll_{\lambda,\epsilon} N^{-\frac{1}{12}+\epsilon}$. We shall improve this estimate further.
\begin{theorem}\label{th:main} Let $f$ be an $L^2$-normalized Hecke--Maass cuspidal newform of square-free level $N$. Then for any $\epsilon>0$ we have a bound
  \[\pnorm{f}_\infty \ll_{\lambda,\epsilon} N^{-\frac{1}{6}+\epsilon},\]
where the implied constant depends continuously on $\lambda$.
\end{theorem}
\begin{remarks*}\begin{enumerate}[(i)]   
	\item It seems that $-\frac{1}{6}$ is the natural exponent for the sup-norm problem in the level aspect. Examples of such exponents are the Weyl exponent $\frac{1}{6}$ (resp. Burgess exponent $\frac{3}{16}$) in the subconvexity problem for $\mathrm{GL}_1$ in the archimedean (resp. nonarchimedean) aspect, or their doubles in the $\mathrm{GL}_2$-setting.
	\item Independently, Blomer--Michel~\cite{BM:hybrid} obtain a bound of the same quality for Hecke eigenforms on unions of arithmetic ellipsoids. In this paper we are concerned in~\eqref{def:count} with solutions of an indefinite quadratic equation $\det(\gamma)=l$, whereas arithmetic ellipsoids involve definite quadratic forms.
  \item	From Atkin--Lehner theory we may assume that $\Mim z\gg N^{-1}$ when investigating the sup-norm. The critical range is actually when $\Mim z \le N^{-\frac{2}{3}+o(1)}$. Otherwise the details of the proof below show that $\abs{f(z)}$ is significantly less than $N^{-\frac{1}{6}}$.
 \end{enumerate}
\end{remarks*}
 
The present note is derived from~\cite{Temp:hybrid} which is motivated by the comparison of the method in~\cite{IS95} for the $\lambda$-aspect with our method in~\cites{Temp:sup,Temp:harcos} for the $N$-aspect. The advantage of the new argument in~\cite{Temp:hybrid} is that it can be adapted to the $\lambda$-aspect to reproduce the bound $\pnorm{f}_\infty \ll_{N,\epsilon} \lambda^{\frac{5}{24}+\epsilon}$, which is key for establishing hybrid bounds simultaneously in the $\lambda$ and $N$-aspects. 
	Compared to~\cites{Temp:sup,Temp:harcos} the reader will find below two improvements coming from a Pell equation and a uniform count of lattice points~\cite{Schm68}.

 \section{Counting lattice points}

 \subsection{Notation}  To make this section self-contained we recall the definitions from~\cites{Temp:sup,Temp:harcos}. Let $\mathrm{GL}_2(\BmR)^+$ act on the upper-half plane $\FmH=\set{x+iy,\ y>0}$ by fractional linear transformations. Denote by $u(,)$ the following function of the hyperbolic distance:
\begin{equation} 
  u(w,z)=\frac{\abs{w-z}^2}{4\Mim w\ \Mim z}.
\end{equation}

For $z\in \FmH$ and $l,N\ge 1$ let $M_*(z,l,N)$ be the number of matrices $\gamma=\Mdede{a}{b}{c}{d}$ in $M_2(\BmZ)$ such that
\begin{equation} \label{def:count}
  \det(\gamma)=l,\quad c\equiv 0(N),\quad u(\gamma z,z)\le N^\epsilon,\quad c\neq 0,\quad (a+d)^2\neq 4l.
\end{equation}

 We write $f \preccurlyeq g$ meaning that for all $\epsilon>0$ there is a constant $C(\epsilon)>0$ such that $f(N)\le C(\epsilon)N^\epsilon g(N)$ for all $N\ge 1$. To simplify notation we omit the dependence in $\lambda$.
For example Theorem~\ref{th:main} says $\pnorm{f}_\infty \preccurlyeq N^{-\frac{1}{6}}$.

Let $\CmF(N)$ be the set of $z\in \FmH$ such that $\Mim z \ge \Mim \delta z$ for all Atkin--Lehner operators $\delta$ of level $N$. In this section we shall only use the fact (\cite{Temp:harcos}*{Lemma~2.2}) that for all $z=x+iy\in \CmF(N)$, we have $ Ny\gg 1$ and that for all $(a,b)\in \BmZ^2$ distinct from $(0,0)$ we have
\begin{equation} \label{shortest}
  \abs{az+b}^2 \ge \frac{1}{N}.
\end{equation}

\subsection{Lattice points}
We have the following uniform estimate for the number of lattice points in a disc (\cite{Schm68}*{Lemma~2}):
\begin{lemma}\label{lem:minkowski} Let $M$ be a euclidean lattice of rank $2$ and $D$ be a disc of radius $R>0$ in $M \otimes_\BmZ \BmR$ (not necessarily centered at $0$). If $\lambda_1\le \lambda_2$ are the successive minima of $M$, then
\begin{equation}\label{mink}
 \# M\cap D \ll 1+ \frac{R}{\lambda_1}+ \frac{R^2}{\lambda_1 \lambda_2}
\end{equation}
\end{lemma}

\begin{remarks*}\begin{enumerate}[(i)] 
  \item Let $d(M)>0$  be the covolume of $M$. Minkowski's second Theorem asserts that $\lambda_1\lambda_2 \asymp d(M)$. When $R\to \infty$, the leading term of~\eqref{mink} is $\frac{R^2}{d(M)}$ as expected.
  \item It is easier to establish the upper-bound $\ll 1+\frac{R^2}{\lambda_1^2}$ (which also has the advantage of having only two terms). One can verify that
	\begin{equation}\label{rem:compare} 
  1+ \frac{R}{\lambda_1}+ \frac{R^2}{\lambda_1 \lambda_2}
  \ll 1+ \frac{R^2}{\lambda_1^2}.
\end{equation}
Thus the estimate in~\eqref{mink} is always better.
\item 
  We have equality (up to a constant) in~\eqref{rem:compare} if and only if $R\ll \lambda_1$. In the applications below it is often the case that $R\ll \lambda_1$. However this is not always the case, and then the improvement of~\eqref{mink} on the easier bound is significant.
\end{enumerate}
\end{remarks*}

\subsection{Counting}
The following is an improvement on~\cite{Temp:harcos}*{Lemma~4.2}:
\begin{lemma}\label{lem:ht}
  Let $z=x+iy \in \CmF(N)$ and $1\le L \le N^{O(1)}$. Then
  \begin{equation} \label{eq:lem:ht}
	\sum_{1\le l\le L} M_*(z,l,N) \preccurlyeq 
	\frac{L}{Ny} + 
	\frac{L^{\frac{3}{2}}}{N^{\frac{1}{2}}}
	+\frac{L^{2}}{N}.
\end{equation}
If we restrict to $l$ being a perfect square, then one can improve by a factor $L^{\frac{1}{2}}$:
\begin{equation}\label{lsquare} 
	\sum_{\substack{1\le l\le L,\\ \text{$l$\ is a square}}}
	M_*(z,l,N) \preccurlyeq
	\frac{L^{\frac{1}{2}}}{Ny} + 
	\frac{L}{N^{\frac{1}{2}}}
	+\frac{L^{\frac{3}{2}}}{N}.
\end{equation}
\end{lemma}
\begin{proof}
We briefly recall the beginning of the argument in~\cite{Temp:harcos}*{Lemma~4.2}.
Let $\gamma=\Mdede{a}{b}{c}{d}$ satisfy~\eqref{def:count}. In coordinates we have
\begin{equation} \label{pf:disc}
  \abs{-cz^2+(a-d)z+b}^2 \le L y^2N^\epsilon.
\end{equation}

As in~\cites{IS95,Temp:harcos} we verify that $\abs{c}\preccurlyeq L^{\frac{1}{2}}/y$, so there are $\preccurlyeq L^{\frac{1}{2}}/(Ny)$ possible values of $c$.

Consider the lattice $\langle 1,z \rangle$ inside $\C$. Its covolume equals $y$ and its shortest length is at least $N^{-1/2}$ by~\eqref{shortest}. In the inequality~\eqref{pf:disc} we are counting lattice points $(a-d,b)$ in a disc of volume $\preccurlyeq Ly^2$ centered at $cz^2$. Hence by Lemma~\eqref{lem:minkowski}, there are $\preccurlyeq 1+ \frac{L^{\frac12}y}{N^{-\frac12}} + \frac{Ly^2}{y}$ possible pairs $(a-d,b)$ for each value of $c$.

As in~\cites{IS95,Temp:harcos} one can deduce from~\eqref{pf:disc} that $\abs{a+d} \preccurlyeq L^{\frac{1}{2}}$. This concludes the proof of~\eqref{eq:lem:ht}.

For~\eqref{lsquare} we instead use the identity
\begin{equation}\label{detidentity}
  (a-d)^2+4bc=(a+d)^2-4l.
\end{equation}
The left-hand side is non-zero by assumption~\eqref{def:count}. Since $l$ is a perfect square, for each given triple $(a-d,b,c)$ the number of pairs $(a+d,l)$ satisfying~\eqref{detidentity} is $\preccurlyeq 1$. This concludes the proof.
\end{proof}

The following is a refinement of~\eqref{lsquare}.
\begin{lemma}\label{lem:new}
  Let $z=x+iy \in \CmF(N)$ and $1 \le l_1 \le \Lambda \le N^{O(1)}$. Then
  \begin{equation} 
	\sum_{1\le l_2\le \Lambda} M_*(z,l_1l_2^2,N)
	\preccurlyeq
	\frac{\Lambda^{\frac{3}{2}}}{Ny} + \frac{\Lambda^3}{N^{\frac{1}{2}}} + \frac{ \Lambda^{\frac{9}{2}}}{N}.
  \end{equation}
\end{lemma}


\begin{proof} Let $\gamma=\Mdede{a}{b}{c}{d}$ satisfy~\eqref{def:count}. We have
\begin{equation} \label{pf:disc1}
  \abs{-cz^2+(a-d)z+b}^2 \le l_1l_2^2 y^2N^\epsilon.
\end{equation}
This implies $\abs{c}\preccurlyeq \Lambda^{\frac{3}{2}}/y$, so there are $\preccurlyeq \Lambda^{\frac{3}{2}}/(Ny)$ possible values of $c$.

For each value of $c$, we again apply~Lemma~\ref{lem:minkowski} to the lattice $\langle 1,z \rangle$ of covolume $y$ and shortest length at least $N^{-\frac{1}{2}}$. In the inequality~\eqref{pf:disc1} we are counting lattice points $(a-d,b)$ in a disc of volume $\preccurlyeq \Lambda^{3} y^2$. This implies that there are $\preccurlyeq 1+\frac{\Lambda^{\frac{3}{2}}y}{N^{-\frac{1}{2}}} + \frac{\Lambda^3 y^2}{y}$ possible pairs $(a-d,b)$ satisfying~\eqref{pf:disc1}.

Further, since $\det(\gamma)=l_1l_2^2$, we have:
\begin{equation}\label{pell} 
  (a-d)^2+4bc = (a+d)^2 - 4l_1l_2^2.
\end{equation}
The left-hand side is already determined by the values of $c$ and $(a-d,b)$. It is nonzero by assumption~\eqref{def:count}. This is a generalized Pell equation in the remaining variables $a+d$ and $l_2$. 

Without loss of generality we can assume that $l_1$ is square-free. One can deduce from~\eqref{pf:disc} that $\abs{a+d} \preccurlyeq \Lambda^{\frac{3}{2}}$. If $l_1=1$ then we are done with a divisor bound as in the proof of~\eqref{lsquare}. 

If $l_1>1$ then we write the solutions of the equation in terms of the fundamental unit. The fundamental unit is always greater than $\frac{1+\sqrt{5}}{2}=1.618\cdots$, which is bounded away from $1$ (for better estimates, see~\cite{FJ11} and the references herein). We deduce that the number of pairs $(a+d,l_2)$ of solutions of~\eqref{pell} is $\ll \Lambda^{o(1)} \preccurlyeq 1$. 

The total number of $\gamma$'s is 
\begin{equation} 
  \preccurlyeq \frac{\Lambda^{\frac{3}{2}}}{Ny} \cdot 
  (1+\Lambda^{\frac{3}{2}}N^{\frac{1}{2}} y+\Lambda^3 y).
\end{equation}
This concludes the proof of the lemma.
\end{proof}

\subsection{Special matrices}
We let $M_u(z,l,N)$ be the number of matrices satisfying~\eqref{def:count} but with the condition $c=0$ instead of $c\neq 0$ (upper-triangular). 
\begin{lemma}\label{lem:Mu}
  Let $z=x+iy\in \CmF(N)$ and $1\le \Lambda \le N^{O(1)}$. Then the following estimates hold, where $l_1, l_2$ run through prime numbers:
 \begin{equation} \label{Mu:2}
	\sum_{1\le l_1,l_2 \le \Lambda}
		M_u(z,l_1l_2,N)
	\preccurlyeq
	\Lambda+  \Lambda^2 N^{\frac{1}{2}}y + \Lambda^{3} y,   
 \end{equation}

  \begin{equation} \label{Mu:3}
	\sum_{1\le l_1,l_2 \le \Lambda}
	M_u(z,l_1l_2^2,N)
	\preccurlyeq
	\Lambda+
	\Lambda^{\frac{5}{2}} N^{\frac{1}{2}} y + \Lambda^4 y,   
 \end{equation}

  \begin{equation} \label{Mu:4}
	\sum_{1\le l_1,l_2 \le \Lambda}
	M_u(z,l_1^2l_2^2,N)
	\preccurlyeq
	1+
	\Lambda^2 N^{\frac{1}{2}} y + \Lambda^4 y.
  \end{equation}
\end{lemma}

\begin{proof}
  We need to count the number of matrices $\gamma=\Mdede{a}{b}{0}{d}\in M_2(\BmZ)$ such that
  \begin{equation} \label{latpoint}
	\abs{(a-d)z+b}^2 \le ad y^2 N^\epsilon
  \end{equation}
and $ad=l_1l_2$ (resp. $ad=l_1l_2^2$, and $ad=l_1^2l_2^2$).

We again consider the lattice $\langle 1,z \rangle$ of covolume $y$ and shortest length at least $N^{-\frac{1}{2}}$. In the inequality~\eqref{latpoint} we are counting lattice points $(a-d,b)$ in a disc of volume $\preccurlyeq ady^2$. 

We consider~\eqref{Mu:2} first. There are $\preccurlyeq 1+ \Lambda N^{\frac12} y + \Lambda^2y$ possible pairs of integers $(a-d,b)$ satisfying~\eqref{latpoint}. Each pair gives rise to $O(\Lambda)$ matrices $\gamma$ (this is because $ad=l_1l_2$).

Next we consider~\eqref{Mu:3}. There are $\preccurlyeq 1+\Lambda^{\frac{3}{2}}N^{\frac12}y+\Lambda^3y$ pairs of integers $(a-d,b)$ satisfying~\eqref{latpoint}.
Each pair gives rise to $O(\Lambda)$ matrices $\gamma$ (this is because $ad=l_1l^2_2$).

Finally we consider~\eqref{Mu:4}. There are $\preccurlyeq 1+\Lambda^{2}N^{\frac12}y+\Lambda^4y$ pairs of integers $(a-d,b)$ satisfying~\eqref{latpoint}.
Since $l_1$ and $l_2$ are primes, we have either $(a=1, d=l^2_1l_2^2)$ or $(a=l_1,d=l_1l_2^2)$ or $(a=l_1^2,d=l_2^2)$, or equivalent configurations. In each configuration, and for a given value of $a-d$, there are $\Lambda^{o(1)}$ pairs $(a,d)$. Thus each pair $(a-d,b)$ gives rise to $\Lambda^{o(1)}$ matrices $\gamma$.
\end{proof}
We note that a similar proof also yields:
\begin{equation}\label{Mu} 
  \sum_{\substack{
  1\le l \le L\\
  \text{ $l$ prime}
  }} M_u(z,l,N) \preccurlyeq 1 + L^{\frac12} N^{\frac{1}{2}}y + Ly.
\end{equation}

Finally let $M_p(z,l,N)$ be the number of matrices satisfying~\eqref{def:count} but instead with the condition $(a+d)^2=4l$ (parabolic) and with no restriction on $c\equiv 0(N)$.
Then~\cite{Temp:harcos}*{Lemma~4.1} gives
\begin{equation} \label{Mp}
  M_p(z,l,N) = 2\delta_\square(l),\quad 1\le l < y^{-2}N^{-\epsilon}.
\end{equation}
Here $\delta_\square(l)=1,0$ depending on whether $l$ is a perfect square or not.

We let $M(z,l,N):=M_*(z,l,N)+M_u(z,l,N)+M_p(z,l,N)$ which is the number of matrices satisfying the first three conditions in~\eqref{def:count}.

\section{Proof of Theorem~\ref{th:main}}

Applying the amplification method of Friedlander--Iwaniec as in~\cite{IS95} and~\cite{Temp:harcos}*{\S3}, we have
\begin{equation} \label{ampli}
  \Lambda^{2} \abs{f(z)}^2 \preccurlyeq \sum_{l\ge 1} \frac{y_l}{\sqrt{l}} M(z,l,N).
\end{equation}
Here $\Lambda^2>0$ is the amplifier length and the sequence $y_l\in \BmR_{\ge 0}$ satisfies: 
\begin{equation*} 
  y_l:= 
  \begin{cases}
	\Lambda,&l=1,\\
	1,&\text{$l=l_1$ or $l_1l_2$ or $l_1l_2^2$ or $l_1^2l_2^2$ with $\Lambda <l_1,l_2<2\Lambda$ primes,}\\
	0,&\text{otherwise}.
  \end{cases}
\end{equation*}


By~\cite{Temp:sup}*{\S3.2}, $\abs{f(x+iy)}\preccurlyeq (Ny)^{-\frac{1}{2}}$. Thus we may assume that $y<N^{-\frac{2}{3}}$ when establishing Theorem~\ref{th:main}. Without loss of generality we can also assume $z=x+iy\in \CmF(N)$.  

We shall choose $\Lambda=N^{\frac{1}{3}-\frac{\epsilon}{4}}$. This implies $\Lambda^4<y^{-2}N^{-\epsilon}$, thus the condition in~\eqref{Mp} is satisfied. Therefore the contribution in~\eqref{ampli} of the {\bf parabolic matrices} is $\ll \Lambda$ using~\eqref{Mp}.

The contribution in~\eqref{ampli} of the {\bf upper-triangular matrices} with $l=1$ is $\preccurlyeq \Lambda (1 + N^{\frac{1}{2}} y+y)$ using~\eqref{Mu}. For $\Lambda < l < 2\Lambda$ it is $\preccurlyeq  \Lambda^{-\frac12}+ N^{\frac{1}{2}} y + \Lambda^{\frac12} y$ using~\eqref{Mu} again. For $\Lambda^2<l<4\Lambda^2$ it is $\preccurlyeq 1+ \Lambda N^{\frac{1}{2}}y + \Lambda^2 y$ using~\eqref{Mu:2} of Lemma~\ref{lem:Mu}. For $\Lambda^3<l<8\Lambda^3$ it is $\preccurlyeq \Lambda N^{\frac{1}{2}}y+\Lambda^{\frac{5}{2}} y$ using~\eqref{Mu:3} of Lemma~\ref{lem:Mu}.  For $l>\Lambda^4$ it is $\preccurlyeq N^{\frac{1}{2}} y+\Lambda^2 y$ using~\eqref{Mu:4} of Lemma~\ref{lem:Mu}.

It now remains to consider the matrices in~\eqref{def:count} counted by $M_*$. The contribution in~\eqref{ampli} of $l=1$ is $\preccurlyeq \Lambda(\frac{1}{Ny}+N^{-\frac{1}{2}})$ using~\eqref{eq:lem:ht} in Lemma~\ref{lem:ht}. For $\Lambda < l < 2\Lambda$ it is $\preccurlyeq \frac{\Lambda^{\frac{1}{2}}}{Ny}+ 
\frac{\Lambda}{N^{\frac{1}{2}}}+
\frac{\Lambda^{\frac{3}{2}}}{N}$ using~\eqref{eq:lem:ht} again.
For $\Lambda^2<l<4\Lambda^2$ it is $\preccurlyeq \frac{\Lambda}{Ny}
+\frac{\Lambda^2}{N^{\frac{1}{2}}}+
\frac{\Lambda^3}{N}$ using~\eqref{eq:lem:ht} again.
For $\Lambda^3<l<8\Lambda^3$ it is $\preccurlyeq \frac{\Lambda}{Ny}+ \frac{\Lambda^{\frac{5}{2}}}{N^{\frac{1}{2}}} + \frac{\Lambda^4}{N}$ using~Lemma~\ref{lem:new}. For $l>\Lambda^4$ it is $\preccurlyeq \frac{1}{Ny}+
\frac{\Lambda^2}{N^{\frac{1}{2}}}
+\frac{\Lambda^4}{N}$ using~\eqref{lsquare}.

Altogether we obtain that
\begin{equation}
  \Lambda^2\abs{f(z)}^2 \preccurlyeq \Lambda + \frac{\Lambda^{\frac{5}{2}}}{N^{\frac{1}{2}}}  +  \frac{\Lambda^4}{N}. 
\end{equation}
Choosing $\Lambda:=N^{\frac{1}{3}-\frac{\epsilon}{4}}$, all three terms above are equal to $N^{\frac{1}{3}+o(1)}$. This concludes the proof of Theorem~\ref{th:main}. \qed

\begin{remark*} The conclusion of Theorem~\ref{th:main} holds true for Hecke--Maass cuspidal newforms of an arbitrary nebentypus and also for holomorphic modular forms. Indeed the amplification method again yields the inequality~\eqref{ampli} above and the rest of the proof goes through without change. Also the assumption that $f$ be a newform is not necessary since Atkin--Lehner theory reduces the general case to the case of newforms. For an oldform $f$, the bound would be in terms of the level from which $f$ was induced.
\end{remark*}

\subsection*{Acknowledgements} We thank the referee for his helpful comments.



\bibliographystyle{plain}

\begin{bibdiv}
\begin{biblist}

\bib{AU95}{article}{
      author={Abbes, Ahmed},
      author={Ullmo, Emmanuel},
       title={Comparaison des m\'etriques d'{A}rakelov et de {P}oincar\'e sur
  {$X\sb 0(N)$}},
        date={1995},
        ISSN={0012-7094},
     journal={Duke Math. J.},
      volume={80},
      number={2},
       pages={295\ndash 307},
}

\bib{BM:hybrid}{article}{
      author={Blomer, V.},
      author={Michel, Ph.},
       title={Hybrid bounds for automorphic forms on ellipsoids over number
  fields},
      eprint={http://arxiv.org/abs/1110.4526},
}

\bib{BH09}{article}{
      author={Blomer, Valentin},
      author={Holowinsky, Roman},
       title={Bounding sup-norms of cusp forms of large level},
        date={2010},
        ISSN={0020-9910},
     journal={Invent. Math.},
      volume={179},
      number={3},
       pages={645\ndash 681},
         url={http://dx.doi.org/10.1007/s00222-009-0228-0},
}

\bib{FJ11}{article}{
      author={Fouvry, E.},
      author={Jouve, F.},
       title={A positive density of fundamental discriminants with large
  regulator},
      eprint={http://www.math.u-psud.fr/~fouvry/},
}

\bib{Temp:harcos}{article}{
      author={Harcos, G.},
      author={Templier, N.},
       title={On the sup-norm of {M}aass cusp forms of large level. {I}{I}},
        date={2011},
     journal={Int. Math. Res. Not.},
       pages={Art. ID rnr202, 11pp.},
}

\bib{IS95}{article}{
      author={Iwaniec, H.},
      author={Sarnak, P.},
       title={{$L\sp \infty$} norms of eigenfunctions of arithmetic surfaces},
        date={1995},
        ISSN={0003-486X},
     journal={Ann. of Math. (2)},
      volume={141},
      number={2},
       pages={301\ndash 320},
}

\bib{MU98}{article}{
      author={Michel, Ph.},
      author={Ullmo, E.},
       title={Points de petite hauteur sur les courbes modulaires {$X\sb
  0(N)$}},
        date={1998},
        ISSN={0020-9910},
     journal={Invent. Math.},
      volume={131},
      number={3},
       pages={645\ndash 674},
}

\bib{Schm68}{article}{
      author={Schmidt, Wolfgang~M.},
       title={Asymptotic formulae for point lattices of bounded determinant and
  subspaces of bounded height},
        date={1968},
        ISSN={0012-7094},
     journal={Duke Math. J.},
      volume={35},
       pages={327\ndash 339},
}

\bib{SS89}{article}{
      author={Seeger, A.},
      author={Sogge, C.~D.},
       title={Bounds for eigenfunctions of differential operators},
        date={1989},
        ISSN={0022-2518},
     journal={Indiana Univ. Math. J.},
      volume={38},
      number={3},
       pages={669\ndash 682},
}

\bib{Temp:hybrid}{article}{
      author={Templier, N.},
       title={Hybrid sup-norm bounds for {H}ecke-{M}aass cusp forms},
     journal={Submitted.},
}

\bib{Temp:sup}{article}{
      author={Templier, N.},
       title={On the sup-norm of {M}aass cusp forms of large level},
        date={2010},
     journal={Selecta Math. (N.S.)},
      volume={16},
      number={3},
       pages={501\ndash 531},
}

\end{biblist}
\end{bibdiv}

\def\cprime{$'$}\def\cprime{$'$}\def\cprime{$'$}\def\cprime{$'$}

\end{document}